\documentclass[11pt]{amsart}

\usepackage{amsmath}
\usepackage{amssymb}
\usepackage{amsfonts}
\usepackage{amsthm}
\usepackage{enumerate}
\usepackage{bbm}
\usepackage[mathscr]{eucal}
\usepackage{graphicx}
\usepackage{verbatim}

\usepackage{hyperref}
\hypersetup{colorlinks=true, urlcolor=black, linkcolor=black, citecolor=black}

\usepackage{dsfont}

\newtheorem{theorem}{Theorem}
\newtheorem{lemma}[theorem]{Lemma}

\newtheorem{proposition}{Proposition}

\theoremstyle{definition}

\theoremstyle{remark}

\begin{document}
\title[Almost everywhere divergence of Bochner-Riesz means]
{On almost everywhere divergence of Bochner-Riesz means on compact Lie groups}
\date{September 1, 2016}
\author{Xianghong Chen and Dashan Fan}
\address{X. Chen\\Department of Mathematical Sciences\\University of Wisconsin-Milwaukee\\Milwaukee, WI 53211, USA}
\email{chen242@uwm.edu}
\address{D. Fan\\Department of Mathematical Sciences\\University of Wisconsin-Milwaukee\\Milwaukee, WI 53211, USA}
\email{fan@uwm.edu}
\thanks{}

\dedicatory{}
\commby{}

\begin{abstract}
Let $G$ be a connected, simply connected, compact semisimple Lie group of dimension $n$. It has been shown by Clerc \cite{Clerc1974} that, for any $f\in L^1(G)$, the Bochner-Riesz mean $S_R^\delta(f)$ converges almost everywhere to $f$, provided $\delta>(n-1)/2$. In this paper, we show that, at the critical index $\delta=(n-1)/2$, there exists an $f\in L^1(G)$ such that 
$$\limsup_{R\rightarrow\infty} \big|S_{R}^{(n-1)/2}(f)(x)\big|=\infty,
\ \text{a.e.}\ x\in G.$$
This is an analogue of a well-known result of Kolmogorov \cite{Kolmogoroff1923} for Fourier series on the circle, and a result of Stein \cite{Stein1961} for Bochner-Riesz means on the tori $\mathbb T^{n}, n\geq 2$. We also study localization properties of the Bochner-Riesz mean $S_{R}^{(n-1)/2}(f)$ for $f\in L^1(G)$.
\end{abstract}
\maketitle

\section{Introduction}
Let $\mathbb T^n$ be the $n$-dimensional torus. For any $f\in L^{1}(\mathbb T^n)$ one can consider its (formal) Fourier series expansion
\begin{equation}\label{intro:fourier-series}
f(x)\sim \sum_{k \in\mathbb{Z}^n} \widehat f(k ) e^{2\pi i \langle k ,x\rangle},
\ x\in\mathbb T^n
\end{equation}
where $\langle k ,x\rangle=k _1 x_1 + \cdots + k _n x_n$, 
$\widehat f(k )=\int_{\mathbb T^n} f(x) e^{-2\pi i \langle k ,x\rangle} dx$ is the $k $-th Fourier coefficient of $f$. To understand the convergence of the Fourier series \eqref{intro:fourier-series}, Bochner \cite{Bochner1936} studied its spherical Riesz means of order $\delta$ defined by
$$S_{R}^{\delta}(f)(x)=\sum_{\left| k \right| <R}
\left(1-\frac{\left| k \right|^{2}}{R^{2}}\right)^{\delta}
\widehat f(k ) e^{2\pi i \langle k ,x\rangle},\ R>1.$$
It is well known that when $\delta>\delta_0=(n-1)/2$, $S_{R}^{\delta}(f)$ converges almost everywhere to $f$ as $R\rightarrow\infty$. In the case $n=1$, $\delta_0=0$ and $S_{R}^{\delta_0}(f)$ becomes the partial sum of the Fourier series \eqref{intro:fourier-series}. A famous result of Kolmogorov \cite{Kolmogoroff1923} states that there exists an $f\in L^1(\mathbb T)$ such that
$$\limsup_{R\rightarrow \infty} |S_{R}^{\delta_0}(f)(x)|=\infty,\ \text{a.e.}\ x\in\mathbb T.$$
Stein \cite{Stein1961} extended Kolmogorov's result to $n\ge 2$ and revealed several new features in multiple dimensions (see \cite{SteinWeiss1971} for an exposition).

For a general compact Lie group $G$ of dimension $n$, by the Peter-Weyl theorem, it is natural to consider, for any $f\in L^1(G)$, the formal Fourier series expansion
\begin{equation}\label{intro:fourier-series-group}
f\sim \sum_{\lambda \in \Lambda } d_{\lambda}\chi_{\lambda }*f
\end{equation}
where $d_{\lambda}$ and $\chi_{\lambda}$ are resp. the dimension and character of the corresponding irreducible representation of $G$. In this paper, we will consider the case where $G$ is noncommutative. The case where $G$ is commutative can be reduced to the case of torus discussed above. We will further assume that $G$ is \textit{connected, simply connected, and semisimple}. 

In this setting, Clerc \cite{Clerc1974} studied the Bochner-Riesz means defined by
$$S_{R}^{\delta}(f)=\sum_{\lambda\in\Lambda:\atop \left|\lambda+\rho\right|<R}
\left(1-\frac{\left| \lambda+\rho\right| ^{2}}{R^{2}}\right)^{\delta}
d_{\lambda}\chi_{\lambda }*f,\ R>1$$
where $\rho$ is half the sum of the positive roots (see Section \ref{sec:preliminaries} for details). Generalizing the result on the torus, he showed that $S_{R}^{\delta}(f)$ converges almost everywhere to $f$, provided $\delta>(n-1)/2$. At the critical index $\delta=(n-1)/2$, Zalo\v{z}nik
\cite{Zaloznik1988} showed that the convergence holds for $f$ belonging to certain block spaces $B_{q}(G)$ strictly contained in $L^{1}(G)$. To the best of our knowledge, it remains an open question whether the convergence holds for all $f\in L^{1}(G)$. Here we give a negative answer, showing that the Kolmogorov-Stein divergence theorem extends to this setting.

\begin{theorem}\label{thm:kolmogorov-L1}
There exists an $f\in L^1(G)$ such that
$$\limsup_{R\rightarrow\infty} \big|S_R^{(n-1)/2}(f)(x)\big|=\infty,
\ \text{a.e.}\ x\in G.$$
\end{theorem}

Recall that Stein \cite{Stein1961} proved the divergence theorem based on the fact that the Bochner-Riesz kernel $K_R^{\delta_0}(x)$ is unbounded almost everywhere, i.e.
$$\limsup_{R\rightarrow\infty} \left |\sum_{\left| k \right| <R}
\left(1-\frac{\left| k \right|^{2}}{R^{2}}\right)^{\frac{n-1}{2}}
e^{2\pi i \langle k ,x\rangle}\right |=\infty,
\ \text{a.e.}\ x\in \mathbb T^n.$$
The situation is quite different in the noncommutative setting. In fact, it follows from estimates of Clerc \cite{Clerc1974} that 
$$\sup_{R>1} \left |\sum_{\lambda\in\Lambda:\atop \left|\lambda+\rho\right|<R}
\left(1-\frac{\left| \lambda+\rho\right| ^{2}}{R^{2}}\right)^{\frac{n-1}{2}}
d_{\lambda}\chi_{\lambda }(x)\right |<\infty,\ \text{a.e.}\ x\in G.$$
This prevents us from adapting directly Stein's argument for $\mathbb T^n, n\ge2$ to prove Theorem \ref{thm:kolmogorov-L1}. However, as shown in Clerc \cite{Clerc1974}, one can use the Poisson summation formula to decompose $K_R^{\delta_0}(x)$ into a sum of a Dirichlet type kernel and a kernel that can be bounded uniformly by an integrable function on $G$. By Young's inequality one can then focus only on the Dirichlet kernel, a situation reminiscent of Kolmogorov's divergence theorem for the circle. 

While Kolmogorov's proof in \cite{Kolmogoroff1923} does not seem to carry over to the current setting, an alternative proof by Stein turns out to be very useful here. In fact, using an idea of Stein in \cite{Stein1961} (see also \cite{SteinWeiss1971}), to prove Theorem \ref{thm:kolmogorov-L1} it suffices to obtain a divergence estimate with $f$ replaced by suitable finitely supported measures. Using a Glivenko-Cantelli theorem, we show that the empirical measures on $G$ provide the desired estimate. These measures have also been used by Kahane \cite{Kahane1960/1961} to give a probabilistic proof of Kolmogorov's divergence theorem.

The paper is organized as follows. In Section \ref{sec:preliminaries} we introduce some notations and give some preliminaries that will be used in the proof of Theorem \ref{thm:kolmogorov-L1}. In Section \ref{sec:transference} we prove Theorem \ref{thm:kolmogorov-L1} assuming a divergence estimate in terms of measures, which is stated as Lemma \ref{lem:kolmogorov-M}. In Section \ref{sec:random-points} we prove Lemma \ref{lem:kolmogorov-M} using empirical measures on $G$. In Section \ref{sec:remarks} we conclude the paper with some remarks on localization.

The notation $A\preceq B$ means $A\leq CB$ for some constant $C>0$ independent of the variables being considered. $C$ denotes a constant depending only on $G$ whose value may change from line to line.

\section{Notation and Preliminaries}\label{sec:preliminaries}

Throughout the paper,  $G$ will be a connected, simply connected, compact semisimple Lie group of dimension $n$.
Let $\mathfrak{g}$ be the Lie algebra of $G$ and $\mathfrak{t}$ the Lie algebra of a fixed maximal torus $T$ of $G$ of dimension $m$.
Let $A$ be a system of positive roots for $(\mathfrak{g},\mathfrak{t})$, so that card$(A)=(n-m)/2$, and let $\rho=\frac{1}{2}\sum_{\alpha\in A}\alpha$.

Let $\langle\cdot,\cdot\rangle$ and $|\cdot|$ denote resp. the inner product and norm  on $\mathfrak{g}$ induced by the negative of the Killing form $B$ on $\mathfrak{g}_{_\mathbb{C}}$,
the complexification of $\mathfrak{g}$. Note that $\langle\cdot,\cdot\rangle$ induces a bi-invariant metric $d$ on $G$.
Since $B$ restricted to $\mathfrak{t}_{_\mathbb{C}}$ is nondegenerate, for any $\lambda \in (\mathfrak{t}_{_\mathbb{C}})^*$
there is a unique\ $\xi_{\lambda}\in\mathfrak{t}_{_\mathbb{C}}$ such that $\lambda (\xi)=B(\xi,\xi_{\lambda }),\ \forall \xi\in \mathfrak{t}_{_\mathbb{C}}$.
We will identify elements in $(\mathfrak{t}_{_\mathbb{C}})^*$ with elements in $\mathfrak{t}_{_\mathbb{C}}$ using this canonical isomorphism.

Let $e$ be the identity in $G$, and $\Gamma=\{\gamma\in \mathfrak{t}, \exp \gamma=e\}$.
The weight lattice is defined as $P=\{\lambda \in \mathfrak{t}: \langle \gamma,\lambda\rangle\in 2\pi\mathbb{Z},\ \forall \gamma\in\Gamma\}$,
and the set of dominant weights is defined as $\Lambda =\{\lambda \in P, \langle\lambda,\alpha\rangle\geq 0,\ \forall \alpha\in A\}$.
Note that $\Lambda $ parametrizes the set of equivalence classes of unitary irreducible representations of $G$. 

For $\lambda \in \Lambda$, let $U_{\lambda }$ be the corresponding representation. By the Weyl character formula, we have
\begin{equation*}
\chi_{\lambda }(\exp \xi )=\text{tr}_{_{U_\lambda}}(\exp \xi)
=\frac{\sum_{w\in W}\varepsilon (w) e^{i\langle w(\lambda +\rho),\xi \rangle}}{D(\xi )},\ \xi\in\mathfrak t,
\end{equation*}
\begin{equation*}
d_{\lambda }=\dim U_\lambda=\prod\limits_{\alpha\in A}\frac{\langle\lambda + \rho, \alpha\rangle}{\langle \rho , \alpha\rangle}
\end{equation*}
where $W$ is the Weyl group which acts on $T$ and $\mathfrak t$, $\varepsilon(w) $ is the signature of $w\in W$, and
\begin{equation}\label{eqn:weyl}
D(\xi )=\sum_{w\in W} \varepsilon(w) e^{i\langle w(\rho) ,\xi \rangle}
=(2i)^{\frac{n-m}{2}}
\prod\limits_{\alpha \in A}\sin\frac{\langle\alpha ,\xi \rangle}{2}
\end{equation}
is the Weyl denominator. Note that $|D(\xi)|$ is $\Gamma$-periodic and $W$-invariant, therefore can also be defined on $G$ by letting $|D(x)|=|D(\exp\xi)|$ whenever $\exp\xi\in T$ is conjugate to $x\in G$.

Denote by $dx$ (resp. $dt$) the normalized Haar measure on $G$ (resp. $T$). For any central function $f$ on $G$, by the Weyl integration formula, we have
\begin{equation}\label{eqn:jacobian}
\int_G f(x)dx=\frac{1}{|W|}\int_T f(t)|D(t)|^2dt.
\end{equation}
Note that the central functions on $G$ can be naturally identified with $W$-invariant functions on $T$. Let
$$U_{\lambda}(x)=\big[a_{i,j}^{\lambda}(x)\big]_{1\le i,j\le d_\lambda}$$
be the matrix coefficients of $U_\lambda$. By the Peter-Weyl theorem,
$$\left\{a_{i,j}^{\lambda}(x): \lambda \in
\Lambda, 1\leq i,j\leq d_{\lambda}\right\}$$
forms a complete orthogonal system in $L^{2}(G)$. In particular, letting
$$b_{i,j}^{\lambda}(x)=\sqrt{d_{\lambda}}a_{i,j}^{\lambda}(x),$$
we get an orthonormal basis in $L^{2}(G)$:
$$\left\{b_{i,j}^{\lambda}(x): \lambda \in \Lambda, 0\leq i,j\leq
d_{\lambda}\right\}.$$

Given an $f\in L^1(G)$, we can consider the Fourier series expansion
\begin{equation}\label{eqn:b-expansion}
f(x)\sim \sum_{\lambda \in \Lambda }\sum_{0\leq i,j\leq d_{\lambda}}\left(
\int_{G}f(y)\overline{b_{i,j}^{\lambda}(y)}dy \right)
b_{i,j}^{\lambda}(x).
\end{equation}
Letting
$$\Phi_{\lambda}(x)=\sqrt{d_{\lambda}} U_{\lambda}(x),$$
we can rewrite \eqref{eqn:b-expansion} as
$$f(x)\sim\sum_{\lambda \in \Lambda }\text{tr}\left(C_{\lambda}\Phi_{\lambda}(x)^{\top}\right)$$
where
$$C_{\lambda}=\int_{G}f(y)\overline{\Phi _{\lambda}(y)}dy.$$
Using the representation property, we can further write
\begin{align*}
\sum_{\lambda \in \Lambda }\text{tr}\left(C_{\lambda}\Phi_{\lambda}(x)^{\top}\right)
&=\sum_{\lambda \in \Lambda } d_{\lambda} \text{tr}\left(\int_{G}f(y)\overline{U_{\lambda}(y)} U_{\lambda}(x)^{\top} dy \right)\\
&=\sum_{\lambda \in \Lambda } \int_{G}f(y) d_{\lambda} \text{tr}\left(\overline{U_{\lambda}(yx^{-1})}\right) dy\\
&=\sum_{\lambda \in \Lambda } \int_{G}f(y) d_{\lambda} \text{tr}\left({U_{\lambda}(xy^{-1})}\right) dy\\
&=\sum_{\lambda \in \Lambda } {d_{\lambda}\chi_{\lambda}}* f(x)
\end{align*}
where we denote
\begin{equation*}
\left( f* g\right)(x)=\int_{G}f(xy^{-1})g(y)dy.
\end{equation*}

Note that, writing
$$\widehat{f}(\lambda)=\frac{1}{\sqrt{d_{\lambda}}}C_{\lambda},$$
we have
\begin{equation*}
\widehat{f* g}(\lambda)=\widehat{f}(\lambda )\widehat{g}(\lambda).
\end{equation*}
Moreover, if $f$ is a central function, then
\begin{equation*}
\widehat{f}(\lambda)
=\left (\frac{1}{d_\lambda}\int_G f(y)\overline{\chi_\lambda(y)} dy\right ) I_{\lambda}
\end{equation*}
where $I_{\lambda}$ is the $d_\lambda\times d_\lambda$ identity matrix. From this one can deduce
\begin{equation}\label{eqn:chi-convolution}
d_{\lambda}\chi_\lambda* d_{\lambda'}\chi_{\lambda'}=
\begin{cases} 
0 & \text{if } \lambda\neq \lambda'\\
{d_\lambda}\chi_\lambda & \text{if } \lambda=\lambda'.
\end{cases}
\end{equation}

The Fourier series \eqref{eqn:b-expansion} has been studied by many authors (cf. \cite{BloomXu1994}, \cite{Cecchini1972}, \cite{ColzaniGiuliniTravaglini1989}, \cite{ColzaniGiuliniTravagliniEtAl1990}, \cite{CowlingManteroRicci1982}, \cite{Dreseler1981}, \cite{GiuliniSoardiTravaglini1982}, \cite{Meaney1978}, \cite{Ragozin1976}, \cite{Stanton1976}, \cite{StantonTomas1978}, \cite{Sugiura1971}, \cite{Taylor1968}, \cite{Xu1992} and references therein). Following \cite{Clerc1974}, we define the Bochner-Riesz means of order $\delta$ by
\begin{equation*}
S_{R}^{\delta}(f)(x)
=\sum_{\lambda \in \Lambda: \atop |\lambda+\rho|<R} \sum_{0\leq i,j\leq d_{\lambda}}
\left( 1-\frac{|\lambda +\rho|^{2}}{R^{2}}\right)^{\delta}
\left( \int_{G}f(y)\overline{b_{i,j}^{\lambda}(y)}dy \right)
b_{i,j}^{\lambda}(x).
\end{equation*}
Equivalently,
\begin{align*}
S_{R}^{\delta }(f)(x)
&=\sum_{\lambda \in \Lambda: \atop |\lambda+\rho|<R}
\left( 1-\frac{|\lambda +\rho|^{2}}{R^{2}}\right)^{\delta}
d_\lambda \text{tr}\left(\widehat f(\lambda)U_\lambda(x)^{\top}\right)\\
&=\sum_{\lambda \in \Lambda: \atop |\lambda+\rho|<R}
\left( 1-\frac{|\lambda +\rho|^{2}}{R^{2}}\right)^{\delta}
d_{\lambda}\chi_{\lambda}* f(x)\\
&=K_{R}^{\delta }* f(x)
\end{align*}
where
\begin{equation*}
K_{R}^{\delta }(x)
=\sum_{\lambda \in \Lambda: \atop |\lambda+\rho|<R}
\left( 1-\frac{|\lambda +\rho|^{2}}{R^{2}}\right)^{\delta }d_{\lambda}\chi_{\lambda}(x).
\end{equation*}
In what follows we will consider
$$\delta=\delta_0=\frac{n-1}{2}.$$

More generally, given a function $\varphi\in C[0,\infty)$ with compact support, let
$$K^\varphi_R(x)=\sum_{\lambda \in \Lambda}
\varphi\left(\frac{|\lambda+\rho|}{R}\right) d_\lambda \chi_\lambda(x).$$
To estimate $K^\varphi_R(x)$ we will use the Poisson summation formula. We will write
$$\widetilde J_{\nu}(r)=\frac{J_\nu(r)}{r^\nu}$$
where $J_\nu(r)$ is the Bessel function of order $\nu$.

\begin{lemma}[{\cite[Theorem~1]{Clerc1974}}]\label{lem:poisson}
Let
$$\phi(r)=2\pi\int_0^\infty \varphi(s) \widetilde J_{\frac{m-2}{2}}(rs) s^{m-1}ds.$$
Assume that there is $\varepsilon>0$ such that
$$\left|\left(\frac{1}{r}\frac{d}{dr}\right)^\ell \phi(r)\right|\preceq r^{-m-\ell-\varepsilon}$$
for all $\ell$ with $0\le\ell\le k=\frac{n-m}{2}$. Then
\begin{equation}\label{eqn:poisson}
K_R^\varphi(\exp\xi)
=\frac{{C}}{D(\xi)} \sum_{\gamma\in\Gamma}
\Big(\prod_{\alpha\in A} {\langle \alpha,\xi+\gamma \rangle}\Big)
R^n \left(\frac{1}{r}\frac{d}{dr}\right)^k\phi(R|\xi+\gamma|).
\end{equation}
\end{lemma}

Consider $\varphi_{(n)}(x)=\varphi(|x|)$ as a function on $\mathbb R^n$. Then
$$\widehat\varphi_{(n)}(\xi)=\int_{\mathbb R^n} e^{-i\xi\cdot x} \varphi(x)dx$$
is equal to (cf. \cite[p.~155]{SteinWeiss1971})
$$\widehat\varphi_{(n)}(|\xi|)=(2\pi)^{n/2} \int_0^\infty \varphi(s)\widetilde J_{\frac{n-2}{2}}(|\xi|s)s^{n-1}ds.$$
Combining this with the formula (cf. \cite[p.~45]{Watson1995})
$$\frac{d}{dr}\widetilde J_\nu(r)=-r\widetilde J_{\nu+1}(r),$$
we get
\begin{align}\label{eqn:bessel-derivative}
\left(\frac{1}{r}\frac{d}{dr}\right)^k\phi(r)
&=(-1)^k 2\pi\int_0^\infty \varphi(s) \widetilde J_{\frac{m+2k-2}{2}}(rs) s^{m+2k-1}ds\\
&=(-1)^k (2\pi)^{-\frac{n-2}{2}} \widehat\varphi_{(n)}(r).\notag 
\end{align}
This allows us to write \eqref{eqn:poisson} as
\begin{equation}\label{eqn:KGG}
K_R^\varphi(\exp\xi)
=\widetilde K_R^\varphi(\exp\xi)+G_R^\varphi(\exp\xi)
\end{equation}
where
\begin{align}\label{eqn:tilde-G}
\widetilde K_R^\varphi(\exp \xi)
=&{C} \mathds 1_{Q_0}(\xi)
\frac{\prod_{\alpha\in A} {\langle \alpha, \xi\rangle}}{D(\xi)}
R^n \widehat\varphi_{(n)}(R| \xi|),\notag\\
G_R^\varphi(\exp \xi)
=&{C} \mathds 1_{Q\backslash Q_0}(\xi)
\frac{\prod_{\alpha\in A} {\langle \alpha, \xi\rangle}}{D(\xi)}
R^n \widehat\varphi_{(n)}(R| \xi|)\\
&+\frac{{C}}{D( \xi)}
\sum_{\gamma\in\Gamma\backslash\{0\}}
\Big (\prod_{\alpha\in A} {\langle \alpha, \xi+\gamma \rangle}\Big)
R^n \widehat\varphi_{(n)}(R| \xi+\gamma|)\notag.
\end{align}
Here and in what follows, we assume that $\xi$ lies in a $W$-invariant fundamental domain $Q\subset\mathfrak t$, and let
$$Q_0=\left \{\xi\in Q: |\langle \alpha,\xi \rangle|< \pi,\ \forall \alpha\in A \right \}.
\footnote{Alternatively, one can take $Q$ to be any fundamental domain centered at $0$, and $Q_0$ a sufficiently small ball centered at $0$.}$$
Note that $\widetilde K_R^\varphi$ and $G_R^\varphi$ are both $W$-invariant. In particular, they can be extended to the whole group $G$.

\begin{lemma}\label{lem:est-G}
Suppose $\varphi$ is as in Lemma \ref{lem:poisson} and satisfies in addition
$$|\widehat\varphi_{(n)}(\xi)|\preceq |\xi|^{-n}.$$
Then
$$\big|G^\varphi_R(x)\big|\preceq \frac{1}{|D(x)|},\ x\in G.$$
\end{lemma}

\begin{proof}
As in \eqref{eqn:tilde-G}, we can write
$$G^\varphi_R(x)=I+II.$$	
To bound $|I|$, it suffices to show that, for all $\xi\in Q\backslash Q_0,$
$$\Big|\prod_{\alpha\in A} {\langle \alpha, \xi\rangle} \Big|
\big|R^n \widehat\varphi_{(n)}(R| \xi|)\big|
\preceq 1.$$
However, this follows immediately from 
$$\Big|\prod_{\alpha\in A} {\langle \alpha, \xi\rangle} \Big|
\preceq |\xi|^{\frac{n-m}{2}}\preceq 1$$
and, by the assumption, 
$$\big|R^n \widehat\varphi_{(n)}(R| \xi|)\big|
\preceq {|\xi|^{-n}}
\preceq 1$$
where the last inequality holds because $\xi\notin Q_0$.
	
To bound $|II|$, it suffices to show that, for all $\xi\in Q$,
$$\sum_{\gamma\in\Gamma\backslash\{0\}}
\Big|\prod_{\alpha\in A} {\langle \alpha, \xi+\gamma \rangle}\Big |
\Big|R^n \widehat\varphi_{(n)}(R| \xi+\gamma|)\Big|\preceq 1.$$
From $\xi\in Q$, $\gamma\in\Gamma\backslash\{0\}$, we get $|\xi+\gamma|\succeq |\gamma|$. On the other hand, we have
$$\Big|\prod_{\alpha\in A} {\langle \alpha, \xi+\gamma \rangle}\Big |\preceq |\xi+\gamma|^{\frac{n-m}{2}},$$
and, by the assumption,
$$\big|R^n \widehat\varphi_{(n)}(R| \xi+\gamma|)\big|\preceq |\xi+\gamma|^{-n}.$$
Since $n-{\frac{n-m}{2}}>m$, the desired bound follows by summing over $\gamma$.
\end{proof}

\begin{lemma}\label{lem:est-K-L1}
Suppose $\varphi$ is as in Lemma \ref{lem:poisson} and satisfies in addition
$$|\widehat\varphi_{(n)}(\xi)|\preceq |\xi|^{-n-\epsilon}$$
for some $\epsilon>0$. Then
$$\sup_{R>1}\big\|K_R^\varphi\big\|_{L^1(G)}<\infty.$$
\end{lemma}

\begin{proof}
By \eqref{eqn:jacobian}, we have
$$\int_G |K_R^\varphi(x)|dx
=\frac{1}{|W|}\int_T |K_R^\varphi(t)||D(t)|^2dt.$$
By \eqref{eqn:poisson}, this can be bounded by
$$\int_Q \sum_{\gamma\in\Gamma}
\Big|\Big (\prod_{\alpha\in A} {\langle \alpha, \xi+\gamma \rangle}\Big)
R^n \widehat\varphi_{(n)}(R| \xi+\gamma|)\Big| |D(\xi)|d\xi.$$
Using $|D(\xi)|=|D(\xi+\gamma)|$ and $|D(\xi)|\preceq |\xi|^{\frac{n-m}{2}}$, we can bound this by
\begin{align*}
&\int_{\mathbb R^m}
\Big|\Big (\prod_{\alpha\in A} {\langle \alpha, \xi \rangle}\Big)
R^n \widehat\varphi_{(n)}(R|\xi|)\Big| |\xi|^{\frac{n-m}{2}} d\xi\\
=&\int_{\mathbb R^m}
\Big|\Big (\prod_{\alpha\in A} {\langle \alpha, \xi \rangle}\Big)
\widehat\varphi_{(n)}(|\xi|)\Big| |\xi|^{\frac{n-m}{2}} d\xi\\
\preceq&\int_{\mathbb R^m}|\xi|^{n-m} (1+|\xi|)^{-n-\epsilon} d\xi\\
<&\infty.
\end{align*}
This completes the proof.
\end{proof}

Let $\nu$ be a finite Borel measure on $G$. The Hardy-Littlewood maximal function $M(\nu)$ is defined by
$$M(\nu)(x)=\sup_{r>0}\frac{|\nu|(B(x,r))}{|B(x,r)|}$$
where
$$B(x,r)=\left\{y\in G: d(y,x)<r \right\}.$$
Note that we have
\begin{equation}\label{eqn:ahlfors}
r^n \preceq |B(x,r)|\preceq r^n,\ x\in G, 0<r<1.
\end{equation}
Note also that  $M(\nu)(x)\ge \|\nu\|,\ x\in G$. 

\begin{lemma}\label{lem:hardy-littlewood}
For any finite Borel measure $\nu$ on $G$, we have
$$\big|\{x\in G: M(\nu)(x)>t\}\big|
\preceq \frac{\|\nu\|}{t},\ t>0.$$
In particular,
$$M(\nu)(x)<\infty,\ \text{a.e.}\ x\in G.$$
\end{lemma}

\begin{proof}
The proof is standard, cf. \cite[Theorem~2.2]{Heinonen2001}.
\end{proof}

\begin{lemma}\label{lem:est-K-maximal}
Suppose $\varphi$ is as in Lemma \ref{lem:poisson} and satisfies in addition
$$|\widehat\varphi_{(n)}(\xi)|\preceq |\xi|^{-n-\epsilon}$$
for some $\epsilon>0$. Then for any finite Borel measure $\nu$ on $G$,
$$\sup_{R>1} \big|\widetilde K_R^\varphi\big|*|\nu|(x)
\preceq M(\nu)(x),\ x\in G.$$
\end{lemma}

\begin{proof}
Note that by \eqref{eqn:weyl}, we have, for all $\xi\in Q_0$,
$$\left|\frac{\prod_{\alpha\in A} {\langle \alpha, \xi\rangle}}{D( \xi)}\right|\preceq 1.$$	
Therefore it always holds that
$$\big |\widetilde K^\varphi_R(\exp\xi)\big| \preceq R^n |\widehat\varphi_{(n)}(R| \xi|)|.$$

Given $R>1$, let $j_0\in\mathbb N$ be the number with
$${2^{j_0}}<R\le {2^{j_0+1}}.$$ 
For $|\xi|< 1/R$, we estimate 
$$R^n |\widehat\varphi_{(n)}(R| \xi|)|
\preceq R^n.$$
For $2^{j-1}/R\le |\xi|< 2^j/R,\ j=1,\cdots,j_0$, by the assumption,
$$R^n |\widehat\varphi_{(n)}(R|\xi|)|
\preceq \frac{R^n}{(R|\xi|)^{n+\epsilon}}
\preceq \frac{1}{2^{j\epsilon}} \left (\frac{R}{2^{j}}\right )^n.$$
For $|\xi|\ge 2^{j_0}/R$, we have
$$R^n |\widehat\varphi_{(n)}(R|\xi|)|
\preceq \frac{R^n}{(R|\xi|)^{n}}
\preceq \left (\frac{R}{2^{j_0}}\right )^n
\preceq 1.$$
Combining these we get
\begin{align*}
R^n |\widehat\varphi_{(n)}(R|\xi|)|
\preceq& 1 + {R^n}\mathds 1_{\{|\xi|< 1/R\}}\\
&+\sum_{j=1}^{j_0} \frac{1}{2^{j\epsilon}} \left (\frac{R}{2^{j}}\right )^n
\mathds 1_{\{2^{j-1}/R\le |\xi|< 2^j/R\}},
\end{align*}
which in turn implies, for $x\in G$,
\begin{align*}
\big |\widetilde K^\varphi_R(x)\big |
\preceq 1+\sum_{j=0}^{j_0} \frac{1}{2^{j\epsilon}} \left (\frac{R}{2^{j}}\right )^n
\mathds 1_{\{d(x,e)< 2^j/R\}}.
\end{align*}
From this we obtain, using \eqref{eqn:ahlfors},
\begin{align*}
\big |\widetilde K^\varphi_R\big |*|\nu|(x)
\preceq \|\nu\| + \sum_{j=0}^{j_0}  \frac{1}{2^{j\epsilon}} M(\nu)(x)
\preceq M(\nu)(x),
\end{align*}
as desired.
\end{proof}

When $\varphi(r)=(1-r^2)^{\delta_0}_+$, we have
$$\phi(r)={C} \widetilde J_{m/2+\delta_0}(r).$$
In particular, the conditions in Lemma \ref{lem:poisson} are satisfied, and
\begin{align*}
\widehat\varphi_{(n)}(r)
={C} \widetilde J_{n-1/2}(r),\ 
|\widehat\varphi_{(n)}(\xi)|
\preceq |\xi|^{-n}.
\end{align*}
Combining with Lemma \ref{lem:est-G} we obtain the following. 

\begin{lemma}[{\cite[Theorem~2]{Clerc1974}}]\label{lem:poisson-delta}
The kernel $K_R^{\delta_0}$ satisfies
$$K_R^{\delta_0}(x)
=\widetilde K_R^{\delta_0}(x)+O\left(\frac{1}{|D(x)|}\right),\ x\in G$$
where $\widetilde K_R^{\delta_0}$ is a central function satisfying
$$\widetilde K_R^{\delta_0}(\exp \xi)
={C} \mathds 1_{Q_0}(\xi)
\frac{\prod_{\alpha\in A} {\langle \alpha, \xi\rangle}}{D( \xi)}
R^n \widetilde J_{n-1/2}(R|\xi|)
,\ \xi\in Q.$$
\end{lemma}

\section{Proof of Theorem \ref{thm:kolmogorov-L1}}\label{sec:transference}
The proof of Theorem \ref{thm:kolmogorov-L1} relies on the following lemma, whose proof will be given in Section \ref{sec:random-points}.

\begin{lemma}\label{lem:kolmogorov-M}
Given $L>1$ and $\varepsilon>0$, there exists a Borel probability measure $\mu$ on $G$ such that
$$\limsup_{R\rightarrow\infty} 
\big|\widetilde K_R^{\delta_0}*\mu(x)\big|>L$$
holds on a set $E\subset G$ with $|G\backslash E|<\varepsilon$.
\end{lemma}

To prove Theorem \ref{thm:kolmogorov-L1}, we introduce two more functions. Let $v(r)$ be a smooth function on $\mathbb R$ satisfying
\begin{alignat*}{2}
v(r)=1,&&\ \text{if }|r|\le 1,\\
v(r)=0,&&\ \text{if }|r|\ge 2.
\end{alignat*}
Let
$$V_R(x)=K^v_R(x)=\sum_{\lambda\in\Lambda} v\left (\frac{|\lambda+\rho|}{R} \right ) d_\lambda\chi_\lambda(x).$$
By the smoothness of $v$, the conditions in Lemma \ref{lem:est-K-L1} are satisfied. So we have
\begin{lemma}\label{lem:V-L1-bound}
$$\sup_{R>1}\big\|V_R\big\|_{L^1(G)}<\infty.$$
\end{lemma}

Let $\bar\varphi$ be a function on $\mathbb R$ which is smooth in $(-1,1)$ and which satisfies
\begin{alignat*}{2}
&\bar{\varphi}(r)=1,&&\ \text{if }|r|\le 1/3,\\
&\bar{\varphi}(r)=(1-|r|^2)_+^{\delta_0},&&\ \text{if }|r|\ge 2/3.
\end{alignat*}
Let
$$\bar K^{\delta_0}_R(x)
=K^{\bar\varphi}_R(x)
=\sum_{\lambda \in \Lambda}
\bar\varphi\left(\frac{|\lambda+\rho|}{R}\right) d_\lambda \chi_\lambda(x).$$
Since $\bar{\varphi}(r)$ and $(1-|r|^2)_+^{\delta_0}$ differ by a compactly supported smooth function, by Lemma \ref{lem:poisson}, \ref{lem:est-G}, \ref{lem:est-K-maximal} and \ref{lem:poisson-delta}, for any finite Borel measure $\nu$, we have
\begin{align*}
\big |(K^{\delta_0}_R-\bar K^{\delta_0}_R)*\nu(x)\big|
&\preceq M(\nu)(x)+\frac{1}{|D(\cdot)|}*|\nu|(x),\\
\big |\bar K^{\delta_0}_R-\widetilde K^{\delta_0}_R\big |*|\nu|(x)
&\preceq M(\nu)(x)+\frac{1}{|D(\cdot)|}*|\nu|(x).
\end{align*}
On the other hand, by \eqref{eqn:jacobian} and \eqref{eqn:weyl} we have 
$$\int_G \frac{1}{|D(x)|}dx=\frac{1}{|W|}\int_T {|D(t)|}dt<\infty.$$
Therefore, by Young's inequality,
$$\frac{1}{|D(\cdot)|}*|\nu|(x)<\infty,\ \text{a.e.}\ x\in G.$$
Combining with Lemma \ref{lem:hardy-littlewood}, we get

\begin{lemma}\label{lem:mollifiers}
(i) For any $f\in L^1(G)$,
$$\sup_{R>1} \left | K^{\delta_0}_R*f(x)-\bar K^{\delta_0}_R*f(x) \right |
<\infty,\ \text{a.e.}\ x\in G.$$
(ii) For any finite Borel measure $\nu$,
$$\sup_{R>1} \big|\bar K^{\delta_0}_R-\widetilde K^{\delta_0}_R\big |*|\nu|(x)
<\infty,\ \text{a.e.}\ x\in G.$$
\end{lemma}

Using \eqref{eqn:chi-convolution} and the definitions of $V_R$ and $\bar K^{\delta_0}_R$, one can easily verify following relations.

\begin{lemma}\label{lem:filtration}
(i) If $R\le R'$, then
$$\bar K^{\delta_0}_R*V_{R'}=\bar K^{\delta_0}_R.$$
(ii) If $R\ge 6R'$, then
$$\bar K^{\delta_0}_R*V_{R'}=V_{R'}.$$
\end{lemma}

We are now ready to prove the theorem.

\begin{proof}[Proof of Theorem \ref{thm:kolmogorov-L1}]
By part $(i)$ of Lemma \ref{lem:mollifiers}, it suffices to
find an $f\in L^1(G)$ such that
$$\limsup_{R\rightarrow\infty} \big|\bar K^{\delta_0}_R*f(x)\big|=\infty,
\ \text{a.e.}\ x\in G.$$
The function $f$ will be taken to be of the form
$$f=\sum_{j=1}^\infty \eta_j V_{R_j}*\mu_j$$
where $\{\eta_j>0\}$ is a suitably chosen summable sequence, $\{R_j>1\}$ is a suitably chosen increasing sequence, and each $\mu_j$ is a suitable Borel probability measure chosen from Lemma \ref{lem:kolmogorov-M}. By Lemma \ref{lem:V-L1-bound} we would have
$$\|f\|_{L^1(G)}\preceq \sum_{j=1}^\infty \eta_j<\infty.$$

We will choose $\eta_j, R_j, \mu_j$ inductively. Set $\eta_1=1/2$, $R_1=2$, and $\mu_1=\delta_e$ (the Dirac delta at $e$). Assume that $\eta_{j-1}, R_{j-1}, \mu_{j-1}$ have been chosen, we now choose $\eta_j, R_j, \mu_j$. First, we take $\eta_j>0$ to be such that
\begin{equation}\label{eta<eta}
\eta_j\le\eta_{j-1}/2
\end{equation}
and such that
\begin{equation}\label{etaK<1}
\eta_j \sup_{1\le R\le R_{j-1}}
\big\|\bar K^{\delta_0}_R\big\|_{L^\infty(G)}\le 1.
\end{equation}
With $\eta_j$ chosen, by Lemma \ref{lem:kolmogorov-M} we can find a probability measure $\mu_j$ such that
$$\limsup_{R\rightarrow\infty} 
\big|\widetilde K^{\delta_0}_R*\mu_j(x)\big|
>2^{j+1}\eta_j^{-1}$$
holds on a set $\widetilde E_j\subset G$ with $|G\backslash \widetilde E_j|<2^{-j-1}.$
With such an $\mu_j$ chosen, we can find $R_j$ large enough so that
$$R_j>6R_{j-1}$$
and so that
\begin{equation}\label{eqn:sup>L}
\sup_{6R_{j-1}<R<R_j} \big|\widetilde K_R^{\delta_0}*\mu_j(x)\big|
>2^{j}\eta_j^{-1}
\end{equation}
holds on a set $E_j\subset \widetilde E_j$ with
$$|G\backslash E_j|<2^{-j}.$$
By induction, this completes our choice of $f$.

Now let 
$$E=\bigcup_{k=2}^\infty\bigcap_{j\ge k} E_j.$$ 
It is easy to see that $|E|=1$. To finish the proof, it suffices to show
$$\limsup_{R\rightarrow\infty} 
\big|\bar K^{\delta_0}_R*f(x)\big|
=\infty,\ \text{a.e.}\ x\in E.$$
Note that $x\in E$ implies $x\in E_j$ for all sufficiently large $j$. Fix such an index $j_0$. For any $R$ satisfying
$$6R_{j_0-1}<R<R_{j_0},$$
by Lemma \ref{lem:filtration} we can write
\begin{align*}
\bar K^{\delta_0}_R*f(x)
&=\sum_{j=1}^{\infty} \eta_j \bar K^{\delta_0}_R*V_{R_j}*\mu_j(x)\\
&=\sum_{j<j_0} \eta_j V_{R_j}*\mu_j(x)
+\eta_{j_0} \bar K^{\delta_0}_R*\mu_{j_0}(x)
+\sum_{j>j_0} \eta_j \bar K^{\delta_0}_R*\mu_j(x)\\
&=I+II+III.
\end{align*}
Notice that, since
$$|I|\le \sum_{j=1}^{\infty} \eta_j |V_{R_j}*\mu_j(x)|$$
and
$$\Big\|\sum_{j=1}^{\infty} \eta_j |V_{R_j}*\mu_j|\Big\|_{L^1(G)}<\infty,$$
for a.e. $x\in E$, $|I|$ is bounded independent of $j_0$. Notice also that, by \eqref{eta<eta} and \eqref{etaK<1},
\begin{align*}
|III|
&\le \sum_{j>j_0} \eta_j \|\bar K^{\delta_0}_R\|_{L^\infty(G)}\\
&\le \Big (\sum_{j>j_0} \eta_j \Big )
\sup_{1\le R\le R_{j_0}} \|\bar K^{\delta_0}_R\|_{L^\infty(G)}\\
&\le 2\eta_{j_0+1} \sup_{1\le R\le R_{j_0}} \|\bar K^{\delta_0}_R\|_{L^\infty(G)}\\
&\le 2.
\end{align*}
To estimate $II$, we write
\begin{align*}
II
&=\eta_{j_0} \widetilde K^{\delta_0}_R*\mu_{j_0}(x)
+\eta_{j_0} \big(\bar K^{\delta_0}_R-\widetilde K^{\delta_0}_R\big)*\mu_{j_0}(x)\\
&=II_a+II_b.
\end{align*}
Notice that
\begin{align*}
\big |II_b\big|
&=\big|\big(\bar K^{\delta_0}_R-\widetilde K^{\delta_0}_R\big)
*\eta_{j_0}\mu_{j_0}(x)\big|\\
&\le \sup_{R>1} \big|\bar K^{\delta_0}_R-\widetilde K^{\delta_0}_R\big |*\nu(x)
\end{align*}
where $\nu=\sum_{j=1}^\infty \eta_{j}\mu_{j}$. By part $(ii)$ of Lemma \ref{lem:mollifiers}, the last expression is finite for $\text{a.e.}\ x\in G$. Thus for a.e. $x\in E$, $|II_b|$ is bounded independent of $j_0$ and $R$. On the other hand, by \eqref{eqn:sup>L} we can find $R\in (6R_{j_0-1},R_{j_0})$ such that
$$\big|\widetilde K_R^{\delta_0}*\mu_{j_0}(x)\big|
> 2^{j_0}\eta_{j_0}^{-1},$$
that is, 
$$|II_a|> 2^{j_0}.$$

Combining these estimates, we see that, for a.e. $x\in E$,
\begin{align*}
\limsup_{R\rightarrow\infty} 
\big|\bar K^{\delta_0}_R*f(x)\big|
&\ge \limsup_{j_0\rightarrow\infty} 2^{j_0}-O_x(1)\\
&=\infty.
\end{align*}
This completes the proof of Theorem \ref{thm:kolmogorov-L1}.
\end{proof}

\section{Proof of Lemma \ref{lem:kolmogorov-M}}\label{sec:random-points}

We now prove Lemma \ref{lem:kolmogorov-M}. The key ingredient is a generalization of the classical Glivenko-Cantelli theorem on empirical measures (cf. \cite{Varadarajan1958}).

\begin{lemma}\label{lem:glivenko-cantelli-varadarajan}
Let $G$ be a compact Lie group and $\{y_j\}_{j=1}^\infty$ a sequence of random points chosen independently and uniformly from $G$. Then, almost surely,
the probability measure $\frac{1}{N}\sum_{j=1}^N\delta_{y_j}$ converges weakly to the Haar measure $dy$, as $N\rightarrow\infty$; that is, for any $f\in C(G)$,
$$\lim_{N\rightarrow\infty}\frac{1}{N}\sum_{j=1}^N f(y_j)
=\int_G f(y) dy.$$
\end{lemma}

In what follows we will always assume that $\{y_j\}_{j=1}^\infty$ is as in Lemma \ref{lem:glivenko-cantelli-varadarajan}.

\begin{lemma}\label{lem:glivenko-cantelli-x}
Almost surely, we have
$$\lim_{N\rightarrow\infty}\frac{1}{N}\sum_{j=1}^N f(x y_j^{-1})
=\int_G f(y)dy$$
for any $f\in C(G)$ and $x\in G$.
\end{lemma}

\begin{proof}
This follows immediately from Lemma \ref{lem:glivenko-cantelli-varadarajan} by taking the function to be $f(x y^{-1})$.
\end{proof}

Let $r_0>0$ be small enough that	
$$B(0,2r_0)=\{\xi\in\mathfrak t: |\xi|< 2 r_0\}\subset Q_0,$$
and that 
$$d(\exp(\xi),e)=|\xi|,\ \forall\xi\in B(0,2r_0).$$ 
For $R_0$ sufficiently large,
let $\psi_{R_0}\in C[0,\infty)$ be a nonnegative continuous function satisfying
\begin{alignat*}{2}
\psi_{R_0}(r)&=R_0^{n},&&\quad \text{if }r\le 1/R_0,\\
\psi_{R_0}(r)&=r^{-n},&&\quad \text{if }1/R_0 \le r \le r_0,\\
\psi_{R_0}(r)&\le r^{-n},&&\quad \text{if }r_0 \le r \le 2r_0,\\
\psi_{R_0}(r)&=0, &&\quad \text{if } r\ge 2r_0.
\end{alignat*}
Let
$$k_{R_0}(x)=\psi_{R_0}(d(x,e)).$$

\begin{lemma}\label{lem:k>log}
$$\int_G k_{R_0}(x)dx\succeq \log {R_0}.$$
\end{lemma}

\begin{proof}
Since $k_{R_0}$ is a central function, by \eqref{eqn:jacobian} we have
\begin{align*}
\int_G k_{R_0}(x)dx
=\frac{1}{|W|}\int_T k_{R_0}(t)|D(t)|^2dt.
\end{align*}
	Changing the variable to $\xi\in Q$, this integral can be bounded below by
\begin{align*}
\int_{1/R_0\le|\xi|\le r_0} |\xi|^{-n} |D(\xi)|^2 d\xi.
\end{align*}
Using \eqref{eqn:weyl} and polar coordinates, this can be bounded below by
\begin{align*}
\int_{1/R_0\le|\xi|\le r_0} |\xi|^{-n} 
\Big |\prod\limits_{\alpha \in A} \langle\alpha ,\xi \rangle\Big |^2 d\xi
&\succeq \int_{1/R_0}^{r_0} r^{-n+(n-m)+(m-1)}dr\\
&= \int_{1/R_0}^{r_0} r^{-1}dr\\
&\succeq \log R_0.
\end{align*}
This completes the proof.
\end{proof}

Combining Lemma \ref{lem:glivenko-cantelli-x} and \ref{lem:k>log}, we get

\begin{lemma}\label{lem:sigma k>log}
Almost surely, we have
$$\lim_{N\rightarrow\infty}\frac{1}{N}\sum_{j=1}^N k_{R_0}(x y_j^{-1})
\succeq \log {R_0}$$
for every $x\in G$.
\end{lemma}

Fix $\varepsilon>0$. By a standard limiting argument, from Lemma \ref{lem:sigma k>log} we get 

\begin{lemma}\label{lem:kolmogorov-1}
Almost surely, there exist $N\ge 1$ and $E\subset G$ such that
$$\frac{1}{N}\sum_{j=1}^N k_{R_0}(xy_j^{-1})\succeq \log {R_0},\ \forall x\in E$$
and $|G\backslash E|<\varepsilon$.
\end{lemma}

Pick $\xi_j(x)\in Q$ so that $\exp(\xi_j(x))$ is conjugate to $xy_j^{-1}$. Note that $|\xi_j(x)|$ is independent of the choice of $\xi_j(x)$. The following lemma will be used in the proof of Lemma \ref{lem:kolmogorov-2}. 

\begin{lemma}\label{lem:kronecker}
Almost surely, for almost every $x\in G$, the numbers 
$\big\{|\xi_j(x)|\big\}_{j=1}^\infty$ 
are linearly independent over $\mathbb Q$.
\end{lemma}

\begin{proof}
By Fubini's theorem, it suffices to show that for every $x\in G$, the numbers 
$\big\{|\xi_j(x)|\big\}_{j=1}^\infty$ 
are almost surely linearly independent over $\mathbb Q$. 

Fix $x\in G$. By a standard limiting argument, it suffices to show that for any $N\ge 2$, the numbers $\big\{|\xi_j(x)|\big\}_{j=1}^N$ 
are linearly dependent over $\mathbb Q$ with probability $0$. To this end, let
\begin{align*}
Z_N=\big\{
&(\xi_1,\cdots,\xi_N)\in Q^N:\\
&|\xi_1|, \cdots, |\xi_N| \text{ are linearly dependent over }\mathbb Q
\big\}.
\end{align*}
By induction on $N$ and Fubini's theorem, it is easy to see that
\begin{equation}\label{eqn:Z=0}
\int_{Q^N} \mathds 1_{Z_N}(\xi_1, \cdots, \xi_N)
d\xi_1\cdots d\xi_N=0.
\end{equation}
Since $\{x y_j^{-1}\}_{j=1}^N$ and $\{y_j\}_{j=1}^N$ are equally distributed, by \eqref{eqn:jacobian} we have
\begin{align*}
&\mathbb P\big(|\xi_1(x)|, \cdots, |\xi_N(x)| \text{ are linearly dependent over }\mathbb Q\big)\\
=&C\int_{Q^N} \mathds 1_{Z_N}(\xi_1, \cdots, \xi_N)
|D(\xi_1)|^2\cdots|D(\xi_N)|^2d\xi_1\cdots d\xi_N.
\end{align*}
But by \eqref{eqn:Z=0}, this integral equals $0$. This completes the proof.
\end{proof}

We are now ready to prove Lemma \ref{lem:kolmogorov-M}. Recall that 
$$\widetilde K_R^{\delta_0}(\exp \xi)
={C} \mathds 1_{Q_0}(\xi)
\frac{\prod_{\alpha\in A} {\langle \alpha, \xi\rangle}}{D( \xi)}
R^n \widetilde J_{n-1/2}(R|\xi|)
,\ \xi\in Q.$$
Note that we have the asymptotic form (cf. \cite[p.~199]{Watson1995}):
\begin{equation}\label{eqn:bessel}
R^n \widetilde J_{n-1/2}(R|\xi|)=\sqrt{\frac{2}{\pi}}\frac{\cos(R|\xi|-\frac{n\pi}{2})}{|\xi|^n}
+O\left (\frac{1}{R|\xi|^{n+1}}\right )
,\ R|\xi|\rightarrow\infty.
\end{equation}

\begin{lemma}\label{lem:kolmogorov-2}
Almost surely, there exist $N\ge 1$ and $E\subset G$ such that
$$\limsup_{R\rightarrow\infty} 
\frac{C}{N}\sum_{j=1}^N \widetilde K_R^{\delta_0}(x y_j^{-1}) 
\succeq \log R_0,\ \forall x\in E$$
and $|G\backslash E|<\varepsilon$.
\end{lemma}

\begin{proof}
Let $N$ and $E$ be as in Lemma \ref{lem:kolmogorov-1}. By Lemma \ref{lem:kronecker}, we may assume that for any $x\in E$ the numbers $\big\{|\xi_j|\big\}_{j=1}^N=\big\{|\xi_j(x)|\big\}_{j=1}^N$ are linearly independent over $\mathbb Q$; in particular, $|\xi_j|\neq 0,\ j=1,\cdots,N$.

Fix $x\in E$. By \eqref{eqn:bessel} we can write
\begin{align*}
\frac{C}{N}\sum_{j=1}^N \widetilde K_R^{\delta_0}(x y_j^{-1})
&=\frac{1}{N}\sum_{j=1}^N 
\eta(\xi_j)\frac{\cos(R|\xi_j|-\frac{n\pi}{2})}{|\xi_j|^n}
+O\left (\frac{1}{RN}\sum_{j=1}^N \frac{1}{|\xi_j|^{n+1}}\right )\\
&=I+II
\end{align*}
where $\eta(\xi)$ is a function satisfying
$$\eta(\xi)\succeq \mathds 1_{Q_0}(\xi).$$
Since $N$ is fixed, we have
$$\lim_{R\rightarrow\infty} II=0.$$
On the other hand, since $\big\{|\xi_j|\big\}_{j=1}^N$ are are linearly independent over $\mathbb Q$, by Kronecker's theorem (cf. \cite[Chapter~23]{HardyWright2008}), 
\begin{align*}
\limsup_{R\rightarrow\infty} I
&=\frac{1}{N}\sum_{j=1}^N \eta(\xi_j)\frac{1}{|\xi_j|^n}\\
&\succeq \frac{1}{N}\sum_{j=1}^N \mathds 1_{Q_0}(\xi_j) \frac{1}{|\xi_j|^n}\\
&\ge \frac{1}{N}\sum_{j=1}^N k_{R_0}(x y_j^{-1})\\
&\succeq\log R_0
\end{align*}
where we have used Lemma \ref{lem:kolmogorov-1} in the last inequality. Combining $I$ and $II$ now gives the desired result.
\end{proof}

\begin{proof}[Proof of Lemma \ref{lem:kolmogorov-M}]
Lemma \ref{lem:kolmogorov-M} now follows from Lemma \ref{lem:kolmogorov-2} by taking
$$\mu=\frac{1}{N}\sum_{j=1}^N \delta_{y_j}$$
and $R_0$ to be sufficiently large. This finishes the proof of Lemma \ref{lem:kolmogorov-M}.
\end{proof}

\section{Remarks on localization}\label{sec:remarks}
In this section, we make two remarks concerning localization properties of the Bochner-Riesz mean $S_R^{{(n-1)}/{2}}(f)$ on $L^1(G)$. The first remark (stated as Proposition \ref{pointwise-localization} below) shows that, as in the case of $\mathbb T^n$ $(n\ge 2)$, pointwise localization fails on $G$ as long as $G$ has rank $m\neq 1$ (for the case $m=1$, cf. Mayer \cite{Mayer1968}, \cite{Mayer1968a}). This follows from the uniform boundedness principle and an observation of Bochner \cite{Bochner1936} and Stein \cite{Stein1961}. The second remark (Proposition \ref{ae-localization} below) shows that, although localization fails in the pointwise sense, it holds in an almost everywhere sense. This a.e. localization property is only valid in the noncommutative setting. It has been shown by Stein and Weiss \cite{SteinWeiss1971} that there exists an $f\in L^1(\mathbb T^n)$ $(n\ge 2)$, supported in an arbitrarily small neighborhood of $0$, such that $S_R^{{(n-1)}/{2}}(f)$ diverges almost everywhere on $\mathbb T^n$. 

\begin{proposition}\label{pointwise-localization}
Suppose $G$ has rank $m\neq 1$. Then there exits an $f\in L^1(G)$ which vanishes in a neighborhood of $e\in G$ and and which satisfies
$$\limsup_{R\rightarrow\infty} \big|S_{R}^{(n-1)/2}(f)(e)\big|=\infty.$$
\end{proposition}
\begin{proof}
Fix a sufficiently small $\varepsilon>0$. Assume for a contradiction that $$\sup_{R>1} \big|S_{R}^{(n-1)/2}(f)(e)\big|<\infty$$ holds for all $f\in L^1(G)$ which vanish in $B(e,\varepsilon)$. Since
$$S_{R}^{(n-1)/2}(f)(e)=\int_{G\backslash B(e,\varepsilon)} K^{(n-1)/2}_R(x^{-1}) f(x)dx,$$
by the Banach-Steinhaus theorem, we then have $$\sup_{R>1,d(x,e)>\varepsilon} \big|K^{(n-1)/2}_R(x)\big|<\infty.$$
On the other hand, by the proof of Lemma \ref{lem:poisson-delta},
\begin{equation*}
K_R^{(n-1)/2}(\exp\xi)
=\frac{{C}}{D(\xi)} \sum_{\gamma\in\Gamma}
\Big(\prod_{\alpha\in A} {\langle \alpha,\xi+\gamma \rangle}\Big)
R^n  \widetilde J_{n-1/2}(R|\xi+\gamma|).
\end{equation*}
To obtain a contradiction we will show that
$$\sup_{\varepsilon<|\xi|<2\varepsilon}\limsup_{R\rightarrow\infty} \big|K_R^{(n-1)/2}(\exp\xi)\big|=\infty.$$
Indeed, using \eqref{eqn:bessel}, we can write
\begin{align*}
K_R^{(n-1)/2}(\exp\xi)
=&I+II\\
=&\frac{{C}}{D(\xi)}\sum_{\gamma\in\Gamma}
\frac{\prod_{\alpha\in A} {\langle \alpha,\xi+\gamma \rangle}}{|\xi+\gamma|^n}
\cos(R|\xi+\gamma|-\frac{n\pi}{2})\\
&+\frac{{C}}{D(\xi)} \sum_{\gamma\in\Gamma} \Big(\prod_{\alpha\in A} {\langle \alpha,\xi+\gamma \rangle}\Big)
O\left (\frac{1}{R|\xi+\gamma|^{n+1}}\right ),\ R\rightarrow\infty.
\end{align*}
It is easy to see that $\lim_{R\rightarrow\infty} II=0$ whenever $D(\xi)\neq 0$. On the other hand, by Lemma 4.15 and Lemma 4.17 of \cite{SteinWeiss1971}, we have
$$
\limsup_{R\rightarrow \infty} |I|
=\frac{{C}}{|D(\xi)|}\sum_{\gamma\in\Gamma} \frac{\big|\prod_{\alpha\in A} {\langle \alpha,\xi+\gamma \rangle}\big|}{|\xi+\gamma|^n},\ \text{a.e. } \xi.
$$
Choose $\gamma_0\in\Gamma\backslash\{0\}$ so that $|\langle \alpha,\gamma_0\rangle|\ge 2|\langle \alpha,\xi\rangle|$ whenever $\langle \alpha,\gamma_0\rangle\neq 0$. Then, bounding
$$|\langle \alpha,\xi+\gamma_0 \rangle|
\ge |\langle \alpha,\gamma_0 \rangle|-|\langle \alpha,\xi\rangle|
\ge \frac{1}{2}|\langle \alpha,\gamma_0 \rangle|,$$
we have
$$\limsup_{R\rightarrow \infty} |I|
\succeq \frac{\big|\prod_{\alpha\in A} {\langle \alpha,\xi+\gamma_0 \rangle}\big|}{|D(\xi)|}\\
\succeq \frac{\big|\prod_{\alpha\in A: \langle \alpha,\gamma_0\rangle=0} {\langle \alpha,\xi \rangle}\big|}{|D(\xi)|},\ \text{a.e. } \xi.$$
By \eqref{eqn:weyl}, we have $D(\xi)=C \prod\limits_{\alpha \in A}\sin\frac{\langle\alpha ,\xi \rangle}{2}$, and so
$$\frac{\big|\prod_{\alpha\in A: \langle \alpha,\gamma_0\rangle=0} {\langle \alpha,\xi \rangle}\big|}{|D(\xi)|}
\succeq \frac{1}{\big|\prod_{\alpha\in A: \langle \alpha,\gamma_0\rangle\neq 0} {\langle \alpha,\xi \rangle}\big|}.$$
Combining these we get
$$\sup_{\varepsilon<|\xi|<2\varepsilon}\limsup_{R\rightarrow\infty} \big|K_R^{(n-1)/2}(\exp\xi)\big|
\succeq\sup_{\varepsilon<|\xi|<2\varepsilon} 
\frac{1}{\big|\prod_{\alpha\in A: \langle \alpha,\gamma_0\rangle\neq 0} {\langle \alpha,\xi \rangle}\big|}
=\infty,$$
as desired.
\end{proof}

\begin{proposition}\label{ae-localization}
If $f\in L^1(G)$ vanishes in an open set $U\subset G$, then
$$\lim_{R\rightarrow\infty} S_{R}^{(n-1)/2}(f)(x)=0,\ \text{a.e. } x\in U.$$
\end{proposition}
\begin{proof}
We will show that
$$\lim_{R\rightarrow\infty} S_{R}^{(n-1)/2}(f)(x)=0$$
whenever $x\in U$ satisfies
\begin{equation}\label{eqn:localizatino-condition}
\frac{1}{|D(\cdot)|}*|f|(x)<\infty.
\end{equation}
Note that by Young's inequality, ${|D(\cdot)|^{-1}}*|f|\in L^1(G)$; in particular \eqref{eqn:localizatino-condition} holds for almost every $x\in U$.

To prove the claim, we write
$$S_{R}^{(n-1)/2}(f)(x)=\int_{G} K^{(n-1)/2}_R(y) f(y^{-1}x)dy.$$
By the Weyl integration formula, this equals
\begin{equation}\label{eqn:general-weyl}
\frac{1}{|W|}\int_{T} K^{(n-1)/2}_R(t) F(x,t) |D(t)|^2 dt
\end{equation}
where
$$F(x,t)=\int_{G/T} f(gt^{-1}g^{-1}x) d[g].$$
By the proof of Lemma \ref{lem:poisson-delta}, we can write \eqref{eqn:general-weyl} as
\begin{align*}
C\sum_{\gamma\in\Gamma}\int_{Q} 
\frac{\prod_{\alpha\in A} {\langle \alpha,\xi+\gamma \rangle}}{D(\xi)} 
R^n  \widetilde J_{n-1/2}(R|\xi+\gamma|)
F(x,\exp \xi) |D(\xi)|^2 d\xi.
\end{align*}
By periodicity, this equals
\begin{equation}\label{eqn:riemann-lebesgue}
C\int_{\mathbb R^m}
\frac{\prod_{\alpha\in A} {\langle \alpha,\xi\rangle}}{D(\xi)} 
R^n  \widetilde J_{n-1/2}(R|\xi|)
F(x,\exp \xi) |D(\xi)|^2 d\xi.
\end{equation}
Since $U$ is an open set, there exists an $\varepsilon>0$ such that $F(x,\exp \xi)=0$ whenever $|\xi|\le\varepsilon$.
Using \eqref{eqn:bessel}, we can then write 
$$\eqref{eqn:riemann-lebesgue}=I+II$$ 
where
\begin{align*}
I=C\int_{|\xi|>\varepsilon}
\cos(R|\xi|-\frac{n\pi}{2})
\frac{\prod_{\alpha\in A} {\langle \alpha,\xi\rangle}}{{|\xi|^n}} 
F(x,\exp \xi) \frac{|D(\xi)|^2}{D(\xi)} d\xi
\end{align*}
and
\begin{align*}
|II|
\preceq\frac{1}{R}\int_{|\xi|>\varepsilon}
\frac{|\prod_{\alpha\in A} {\langle \alpha,\xi\rangle}|}{|\xi|^{n+1}} 
|F(x,\exp \xi)| |D(\xi)| d\xi,\ R\rightarrow\infty.
\end{align*}
Applying the Riemann-Lebesgue lemma to $I$, we see that
$$\lim_{R\rightarrow\infty} S_{R}^{(n-1)/2}(f)(x)=0$$
follows if 
$$\int_Q |F(x,\exp \xi)| |D(\xi)|  d\xi<\infty.$$
However, by the Weyl integration formula and \eqref{eqn:localizatino-condition},
\begin{align*}
\int_Q |F(x,\exp \xi)| |D(\xi)| d\xi
&=\int_Q \frac{|F(x,\exp \xi)|}{|D(\xi)|} |D(\xi)|^2 d\xi\\
&\preceq\int_T \int_{G/T} \frac{|f(gt^{-1}g^{-1}x)|}{|D(t)|} d[g] |D(t)|^2 dt\\
&=\int_G \frac{|f(y^{-1}x)|}{|D(y)|}dy\\
&=\frac{1}{|D(\cdot)|}*|f|(x)\\
&<\infty.
\end{align*}
This completes the proof.
\end{proof}

Finally, we remark that it is possible to generalize the results in this paper to the $n$-dimensional spheres $S^n$, $n\ge 2$; 
we hope to address this in future work. We also remark that, even though pointwise convergence may fail, the Bochner-Riesz mean of  $f\in L^1$ always converges to $f$ in measure; see Christ and Sogge \cite{ChristSogge1988}.\\

\noindent\textbf{Acknowledgment.} The second author is partially supported by the National Natural Science Foundation of China (No. 11471288).

\bibliographystyle{abbrv}
\bibliography{bibliography}
\nocite{Stein1958, Stein1983, Clerc1987, Zygmund2002, Katznelson2004,BonamiClerc1973}

\end{document}